\documentclass{article}

\usepackage{amsmath}
\usepackage{amssymb}

\usepackage{amsthm}

\usepackage{xspace}
\usepackage{mathbbol}
\usepackage{tikz}

\newtheorem{theorem}{Theorem}[section]
\newtheorem{lemma}[theorem]{Lemma}

\newtheorem{conjecture}[theorem]{Conjecture}

\title{The Union-Closed Sets Conjecture for Small Families}

\author{Jens Ma{\ss}berg
\footnote{Institut f{\"u}r Optimierung und Operations 
Research, University of Ulm, jens.massberg@uni-ulm.de
}
}

\begin{document}
 \maketitle


 \begin{abstract}
  We prove that the union-closed sets conjecture is true for separating 
  union-closed families $\mathcal{A}$ with
  $|\mathcal{A}| \leq 
  2\left(m+\frac{m}{\log_2(m)-\log_2\log_2(m)}\right)$ where $m$ denotes 
  the number of elements in $\mathcal{A}$.
 \end{abstract}

 \section{Introduction}
 
 A family $\mathcal{A}$ of sets is said to be \emph{union-closed} if for any 
two member sets $A,B\in\mathcal{A}$ their union $A\cup B$ is also a member of  
$\mathcal{A}$. 

A well-known conjecture is the \emph{Union-Closed Sets 
Conjecture} which is also called \emph{Frankl's conjecture}:
 
 \begin{conjecture}
  Any finite non-empty union-closed family of sets has an element that 
is contained in at least half of its member sets.
 \end{conjecture}

There are many papers considering this conjecture. 
So it is known to be true if $\mathcal{A}$ has at most 12 
elements \cite{ziv} or at most 50 member sets \cite{lofaro,robertssimpson} or 
if the number of member sets is large compared to the number $m$ of elements, 
that is  $|\mathcal{A}| \geq \frac{2}{3}2^m$ \cite{balla}. 
Nevertheless, the conjecture is still 
far from being proved or disproved.
 A good survey on the current state of this conjecture is given by Bruhn and 
Schaudt~\cite{BruhnSchaudt}.

In this paper we consider the case that the number of 
member-sets is small compared to the number of elements.
But first we recall some basic definitions and results. Let $\mathcal{A}$ be 
a union-closed set.
We call  $U(\mathcal{A})=\bigcup_{A\in \mathcal{A}} A$ the \emph{universe} 
of  $\mathcal{A}$.
For an element $x\in U(\mathcal{A})$ the cardinality of $|\{A\in\mathcal{A}:\, 
x\in A\}|$ is called the \emph{frequency} of $x$. Thus the union-closed sets 
conjecture states that there exists an element $x\in U(\mathcal{A})$ of 
frequency at least $\frac{1}{2}|\mathcal{A}|$.

A family $\mathcal{A}$ is called \emph{separating} if for any two distinct 
elements $x,y\in U(\mathcal{A})$ there exists a set $A\in\mathcal{A}$ that 
contains exactly one of the elements $x$ and $y$. 
We can restrict ourselves to separating union-closed families: If there exist 
elements $x$ and $y$ such that each member set $A\in\mathcal{A}$ that contains 
$x$ also contains $y$, then we can delete $x$ from each such set and obtain a 
new family of the same cardinality that is still union-closed.
Falgas-Ravry showed 
that there are some sets in  $\mathcal{A}$ satisfying certain conditions 
which help us to analyze small separating union-closed families:
 
 \begin{theorem}[Falgas-Ravry \cite{FalgasRavry}]\label{theorem:fr}
  Let $\mathcal{A}$ be a separating union-closed family and let $x_1,\ldots, 
  x_m$ be the elements of $U(\mathcal{A})$ labeled in order of increasing 
  frequency.
  Then there exist sets $X_0,\ldots, X_m \in\mathcal{A}$ such that
  \begin{equation}
    x_i \notin X_i \quad \forall i\in\{1,\ldots, m\}
  \end{equation}
  and
  \begin{equation}
   \{x_{i+1},\ldots, x_m\} \subset X_i \quad \forall i\in\{0,\ldots, m\}
  \end{equation}
 \end{theorem}
 \begin{proof}
  As $\mathcal{A}$ is separating, for any $1\leq i<j\leq m$ there exists a set 
  $X_{ij}\in\mathcal{A}$ such that $x_i\notin X_{ij}$ and $x_j\in X_{ij}$. For 
  all $1\leq i \leq m-1$ let $X_i=\bigcup _{j=i+1}^m X_{ij}$ and set 
  $X_0=U(\mathcal{A})$.
 \end{proof}

The previous theorem directly implies that the conjecture is satisfied for 
small families:

\begin{lemma}
 Any separating family on $m$ elements with at most $2m$ 
 member sets satisfies the Union-Closed Sets Conjecture.
\end{lemma}
\begin{proof}
 Consider the sets $X_0,\ldots, X_{m-1}$ constructed in Theorem
\ref{theorem:fr} and observe that the most frequent element $x_m$ is contained 
in all these sets. As these sets are pairwise different, $x_m$ is 
contained in at least $m$ of all member sets of $\mathcal{A}$.
\end{proof}

In this paper we show that the Union-Closed Sets Conjecture is also satisfied 
for families that contain (slightly) more then $2m$ member sets.
 Considering such families is motivated by a result of Hu (see also 
\cite{BruhnSchaudt}): 
\begin{theorem}[Hu \cite{Hu}]
   Suppose there is a $c>2$ so that any separating union-closed family 
  $\mathcal{A}'$ with $|\mathcal{A}'|\leq c|U(\mathcal{A}')|$ satisfies the 
  Union-Closed Sets Conjecture. Then, for every union-closed family 
  $\mathcal{A}$, there is an element $x\in U(\mathcal{A})$ of frequency
  \begin{equation}
   |\{A\in\mathcal{A}:\, x\in A\}|\geq \frac{c-2}{2(c-1)}|\mathcal{A}|.
  \end{equation}
 \end{theorem}

Therefore, if the Union-Closed Sets Conjecture is satisfied for 'small' 
families, then for any union-closed family there exists an element that 
appears with a frequency at least a constant fraction of the number of 
member sets.
In this paper we push the bound over $2m$, but for increasing $m$ it still 
converges slowly towards $2m$.
 

 
 \section{Frankl's Conjecture for Small Families}
 
 Combining and extending the idea of the proof of Theorem \ref{theorem:fr} and 
an argument of Knill \cite{knill}  we get the main result of this paper.
 
\begin{theorem}
  The Union-Closed Sets Conjecture is true for separating union-closed families 
 $\mathcal{A}$ with a universe containing $m$ elements 
 satisfying $$|\mathcal{A}| \leq 2\left(m + 
 \frac{m}{\log_2(m)-\log_2\log_2(m)}\right). $$
\end{theorem}

 \begin{proof}
%
 Let $\mathcal{A}$ be a separating union-closed family,
 let the elements $x_1,\ldots, x_m$  of $U(\mathcal{A})$ be labeled in order of
increasing frequency and set $n=|\mathcal{A}|$.
Assume that each element appears in at most $m+c$ member sets. We compute an 
upper bound on the size of $n$.

For $i\in \{1,\ldots, m\}$ we set
\begin{equation}\label{mi}
 M_i = \bigcup_{A\in \mathcal{A}: x_i\notin A} A
\end{equation}
to be the union of all sets containing $x_i$
and we set $M_0=U$.
If the sets $X_i$, $i\in\{0,\ldots, m\}$, are chosen as in Theorem 
\ref{theorem:fr}, then we have $X_i\subset M_i$ for all $i\in 
\{0,\ldots, m-1\}$ and thus 
\begin{equation}\label{eq1}
 \{x_{i+1},\ldots, x_m\} \subseteq M_i.
\end{equation}

Let $\tilde{U}=\{x_i:\, \exists A\in\mathcal{A}\text{ with } \max_{x_j\in A} 
j\}$ be the set of all $x_i$ which are the elements with the highest index in 
some set $A$.

For $x_i\in \tilde{U}$ we set
\begin{equation}
 A_i = \bigcup_{A\in \mathcal{A}:\, i= \max \{j:\, x_j\in A\} } A.
\end{equation}
By definition $x_i\in A_i$. Now consider $j> i$. As $x_j\notin A_i$ we have 
$A_i\subset M_j$. Together with (\ref{eq1}) we have
\begin{equation}
 x_i \in M_j \quad \forall x_i\in \tilde{U}, j\in \{0,\ldots, m-1\}, i\neq j.
\end{equation}

Observe that every non-empty member set of $\mathcal{A}$ touches $\tilde{U}$.
Following an argument of Knill \cite{knill} let $\hat{U}\subseteq\tilde{U}$ be 
minimal such that every non-empty set of 
$\mathcal{A}$ touches $\hat{U}$. Then for all $x_i\in\hat{U}$ there exists a 
set $A\in\mathcal{A}$ with $\hat{U}\cap A = \{x_i\}$; if not, 
$\hat{U}\setminus\{x_i\}$ still touches every member set of $\mathcal{A}$ 
contradicting the minimality of $\hat{U}$.
Therefore as $\mathcal{A}$ is union-closed,
for each $B\subseteq \hat{U}$ there exists a set $P_B\in\mathcal{A}$ with 
$P_B\cap \hat{U}=B$. Let $\mathcal{P}=\{P_B:\, B\subseteq \hat{U}\}$.
The sets in $\mathcal{P}$ are pairwise disjoint and each element 
$x_i\in\hat{U}$ is contained in exactly half of the sets. Setting $k=|\hat{U}|$,
we conclude that there are $2^k$ sets in $\mathcal{P}$ containing in total 
$k2^{k-1}$ elements from $\hat{U}$.

Note, that $\mathcal{P}$ might contain the sets $M_i$ for $x_i\in\hat{U}$
and one additional set $M_j$ with $\hat{U}\subset M_j$.
But then $\{M_0,\ldots, M_{m-1}\}$ contains $m-k$ sets that are not in 
$\mathcal{P}$ and each of these sets contains all elements of $\hat{U}$.

Before we compute an upper bound for the number of elements in $\mathcal{A}$
we summarize the previous observations:

\begin{itemize}
 \item Each of the $k$ elements in $\hat{U}$ appears in at most $m+c$ 
member sets,
 \item the $2^k$ sets in $\mathcal{P}$ contain in total $k2^{k-1}$ copies of 
elements of $\hat{U}$,
 \item there are $m-k$ additional member sets, each containing all elements 
of $\hat{U}$ and
 \item all remaining member sets contain at least one element of $\hat{U}$.
\end{itemize}

We conclude:

\begin{eqnarray}
 n &\leq& k(m+c) + (2^k -k2^{k-1}) + (m-k)(1-k)\\
 &=& m +kc +(2-k)2^{k-1} + k^2 -k. 
\end{eqnarray}

Suppose the Union-Closed Sets Conjecture is wrong, that is, $n>2(m+c)$ or 
$\frac{n}{2}-m>c$.
Then
\begin{eqnarray}
 n & \leq&  m +k(\frac{n}{2}-m) + (2-k)2^{k-1} + k^2-k
\end{eqnarray}
or
\begin{eqnarray}
 n & \geq & 2\frac{(k-1)m+(k-2)2^{k-1}
 +k-k^2
 }{k-2}\\
 &\geq& 2\left(m+2^{k-1} +\frac{m}{k-2} 
 -k-3
 \right).
\end{eqnarray}

We conclude that the conjecture is true for all $n$ satisfying
\begin{equation}
 n  \leq   2\left(m +  \min_{k\in\mathbb{N}} \left( 2^{k-1} 
+\frac{m}{k-2} -k-3\right) 
\right). \label{ieq1}
\end{equation}

The function $f_m(k):=  2^{k-1} 
+\frac{m}{k-2} -k-3$ is convex. \v{Z}ivkovi\'{c} et al.~\cite{ziv} showed
that the Union-Closed Sets Conjecture is satisfied for $m\leq 12$ so we 
can assume that $m\geq 13$. In this case the minimum of $f_m(k)$ is obtained in 
the interval $[5,\log_2(m)]$ and we get

\begin{eqnarray}
 f_m(k) &=& \max \left\{2^{k-1}, \frac{m}{k-2}\right\} + \left( 
\min\left\{2^{k-1}, \frac{m}{k-2}\right\} -3-k \right) \\
 &\geq& \max \left\{2^{k-1}, \frac{m}{k-2}\right\} \\
 &\geq& \min_{k'} \left( \max \left\{2^{k'-1}, \frac{m}{k'-2} \right\} \right) 
\\
 &\geq& \max_{k'} \left( \min \left\{2^{k'-1}, \frac{m}{k'-2} \right\} 
\right).\label{ieq2}
\end{eqnarray}
The last inequality is due to the fact that $2^{k-1}$ is increasing in $k$ 
while $\frac{m}{k-2}$ is decreasing in $k$.

Setting $k'=\log_2(m) -\log_2\log_2(m) +2$
we get
\begin{eqnarray*}
 \log_2\left( \frac{m}{k'-2}\right) &=& \log_2(m) - \log_2
 \left(  \log_2(m)-\log_2\log_2(m)  \right) \\
&=& \log_2(m) -\log_2\log_2(m) - \log_2\left( 
1-\frac{\log_2\log_2(m)}{\log_2(m)}\right) \\
&\leq& \log_2(m) -\log_2\log_2(m) +1  \\
&=&\log_2(2^{k'}).
\end{eqnarray*}

Inserting this result in (\ref{ieq2}) and (\ref{ieq1}) we finally obtain
that the Union-Closed Sets Conjecture is true for all $n$ satisfying
\begin{equation}
 n \leq 2\left(m+\frac{m}{\log_2(m)-\log_2\log_2(m)}\right).
\end{equation}
\end{proof}

\section{Acknowledgement}
 The author thanks Henning Bruhn-Fujimoto for pointing him to the union-closed 
sets conjecture. 

%
%
%
%

\nocite{*}

\bibliography{ucs}{}

\begin{thebibliography}{1}

\bibitem{balla}
I.~Balla, B.~Bollob\'{a}s, and T.~Eccles.
\newblock Union-closed families of sets.
\newblock {\em J. Combin. Theory (Series A)}, 120:531--544, 2013.

\bibitem{BruhnSchaudt}
H.~Bruhn and O.~Schaudt.
\newblock The journey of the union-closed sets conjecture.
\newblock {\em Graphs and Combinatorics}, 2015.
\newblock DOI: 10.1007/s00373-014-1515-0.

\bibitem{FalgasRavry}
V.~Falgas-Ravry.
\newblock Minimal weight in union-closed families.
\newblock {\em Electron. J. Comb.}, 19(P95), 2011.

\bibitem{lofaro}
G.~Lo Faro.
\newblock Union-closed sets conjecture: Improved bounds.
\newblock {\em J. Combin. Math. Combin. Comput.}, 16:97--102, 1994.

\bibitem{Hu}
Y.~Hu.
\newblock Master's thesis (in preperation).

\bibitem{knill}
E.~Knill.
\newblock Graph generated union-closed families of sets, 1994.
\newblock arXiv:math/9409215v1 [math.CO].

\bibitem{robertssimpson}
I.~Roberts and J.~Simpson.
\newblock A note on the union-closed sets conjecture.
\newblock {\em Australas. J. Combin.}, 47:265--267, 2010.

\bibitem{ziv}
M.~\v{Z}ivkovi\'{c} and B.~Vu\v{c}kovi\'{c}.
\newblock The 12-element case of {F}rankls conjecture.
\newblock (submitted, 2012).

\end{thebibliography}
\bibliographystyle{plain} 

\end{document}